\documentclass[letterpaper, 10pt]{amsart}
\usepackage[colorlinks=true, linkcolor=black, citecolor=blue, urlcolor=blue]{hyperref}
\usepackage[all]{xypic}
\usepackage{rotating}
\usepackage{lscape}
\usepackage{amsbsy}
\usepackage{verbatim}
\usepackage{pdfpages}
\usepackage{stmaryrd}
\usepackage{xcolor}
\usepackage{xargs} 
\usepackage{float}%to enable H 
\usepackage{xcolor}

\newcommand{\tmftwo}{E_2^{hG_{24}}}

\newcommand{\dtmf}{E^{hG_{48}}}
\newcommand{\dtmfg}{G_{48}}
\renewcommand{\SS}{\mathbb{S}}
\newcommand{\Gto}{\GG_2^1}
\newcommand{\eG}{E^{h\mathbb{G}_2^1}}
\newcommand{\tmfsd}{E^{hSD_{16}}}
\newcommand{\GGop}{\GG^{\mathrm{op}}\times \GG}

\usepackage[utf8]{inputenc}
\usepackage{amsmath}
\usepackage{amsfonts}
\usepackage{amssymb}
\usepackage{amsthm}
\usepackage{comment}
\usepackage{epsfig}
\usepackage{graphicx}
\usepackage{psfrag}
\usepackage{mathrsfs}
\usepackage{amscd}
\usepackage{csquotes}
\usepackage[all]{xy}
\usepackage{rotating}
\usepackage{lscape}
\usepackage{amsbsy}
\usepackage{verbatim}
\usepackage{moreverb}
\usepackage{tikz}
\usepackage{pdfpages}
\usepackage{stmaryrd}
\usepackage{textcomp}
\usepackage{wasysym}
\usepackage{pgf, pgffor}
\definecolor{lightgray}{gray}{0.8}

\newcommand{\kappabar}{\bar{\kappa}}

\newcommand{\br}[1]{\llbracket #1 \rrbracket}

\DeclareMathOperator{\Gal}{Gal}

\newcommand{\tmf}{E^{hG_{24}}}
\newcommand{\G}{\mathbb{G}_2}

\newcommand{\tmfg}{G_{24}}

\newcommand{\eS}{E^{h\mathbb{S}_2^1}}
\newcommand{\Sto}{\mathbb{S}_2^1}

\newcommand{\xra}{\xrightarrow}
\newcommand{\ra}{\rightarrow}
 
\newcommand{\FF}{\mathbb{F}}
\newcommand{\GG}{\mathbb{G}}
\newcommand{\ZZ}{\mathbb{Z}}

\theoremstyle{plain}
\newtheorem{thm}{Theorem}
\newtheorem{theorem}[thm]{Theorem}
\newtheorem{cor}[thm]{Corollary}
\newtheorem{corollary}[thm]{Corollary}
\newtheorem{lemma}[thm]{Lemma}
\newtheorem{prop}[thm]{Proposition}
\newtheorem{proposition}[thm]{Proposition}

\newtheorem*{thmm}{Theorem 1}

\theoremstyle{definition}

\theoremstyle{remark}

\author[Irina Bobkova]{Irina Bobkova}
\address{Department of Mathematics, Texas A\&M University, College Station, TX, 77843}
\email{ibobkova@tamu.edu}
\title[Spanier--Whitehead duality in the $K(2)$-local category at $p=2$]{Spanier--Whitehead duality in the $K(2)$-local category at $p=2$}
\begin{document}
\begin{abstract}
The fixed point spectra of Morava $E$-theory $E_n$ under the action of finite subgroups of the Morava stabilizer group $\GG_n$, and their $K(n)$-local Spanier--Whitehead duals can be used to approximate the $K(n)$-local sphere in certain cases.
  For any finite subgroup $F$ of $\mathbb{G}_2$ at $p=2$ we prove that the $K(2)$-local Spanier--Whitehead dual of the spectrum $E_2^{hF}$ is $\Sigma^{44}E_2^{hF}$.
  These results are analogous to the known results at height 2 and $p=3$.
  The main computational tool we use  is the topological duality resolution spectral sequence for the spectrum $E_2^{h\mathbb{S}_2^1}$ at $p=2$. 
\end{abstract}
\maketitle
\tableofcontents

The Spanier--Whitehead dual $DX$ of a spectrum $X$ is defined as the function spectrum 
$$
DX =F(X, S^0).
$$
In chromatic homotopy theory we are interested in the categories of spectra localized with respect to Morava $K$-theories $K(n)$, for a fixed prime $p$. For a $K(n)$-local spectrum $X$, it is natural to consider the local Spanier--Whitehead dual, which is the function spectrum in the $K(n)$-local category
\[
   DX=F(X, L_{K(n)}S^0).
\]
Some of the most important $K(n)$-local spectra are $L_{K(n)}S^0$ itself and spectra which approximate it. The $K(n)$-local sphere can be thought of as the homotopy fixed points of Morava $E$-theory $E_n$ under the action of Morava stabilizer group $\GG_n$ (see \cite{DH}). At $n=1$ and $2,$ for appropriate finite subgroups $H\subset \GG_n$  the $K(n)$-local sphere can be decomposed in terms of $E_n^{hH}$ as in, for example, \cite{H}, \cite{GHMR}, \cite{BG}, and \cite{Beh}, hence $E_n^{hH}$ are important building blocks of the $K(n)$-local category. 

While it is, of course, true that 
\[
DE_n^{h\GG_n}=F(E_n^{h\GG_n}, E_n^{h\GG_n})\simeq E_n^{h\GG_n},
\]
it turns out that determining $DE_n^{hH}$ is more complicated for finite and other closed subgroups $H\subset \GG_n,$ even already at chromatic height $n=1.$

For example, at $p=2$ and $n=1$ the maximal finite subgroup of $\GG_1\cong \ZZ_2\times C_2$ is $C_2$ 
and the homotopy fixed point spectrum $E_1^{hC_2}$ fits into a fiber sequence \cite{B}
\[
   L_{K(1)}S^0 \to E_1^{hC_2} \to E_1^{hC_2}.
\]
The $K(1)$-local Spanier--Whitehead dual of $E_1^{hC_2}$ is shown to be \cite{HahnMitchell}
\[
   DE_1^{hC_2}=F(E_1^{hC_2}, L_{K(1)}S^0)\simeq \Sigma^{-1}E_1^{hC_2}
\]
thus allowing us to rewrite the fiber sequence above as
\[
  D E_1^{hC_2} \to L_{K(1)}S^0 \to E_1^{hC_2}.
\]
Now let $n=2$ and $p=3$. There exists a fiber sequence (see \cite{Beh})
\begin{equation}\label{eq:Behrens-fiber-seq}
   DQ(2) \to L_{K(2)} S^0 \to Q(2) 
\end{equation}
 where $Q(2)$ is built from $E_2^{h\tmfg}$ and $E_2^{hD_8}$ ($\tmfg$ is the maximal finite subgroup of $\GG_2$ and $D_{8}$ is another finite subgroup isomorphic to the dihedral group of order 8, see \cite{GHMR} for details on structure of these subgroups). Hence $E_2^{h\tmfg}, E_2^{hD_8}$  and their Spanier--Whitehead duals can be thought of as building blocks for $L_{K(2)}S^0$ at $p=3.$ In \cite{Beh} it is proved that
 \begin{align*}
            D\tmftwo&\simeq \Sigma^{44}\tmftwo\\
             DE_2^{hD_8}  &\simeq \Sigma^{44} E_2^{hD_8}.                                                                                                                                                                                                                                                                                                                                                                                                                                                                                                                                                                                                                                                                                                                                                                                                                                  \end{align*}                                                                                                                                                                                                                                                                                                                                                                                                                   
 It was conjectured (\cite{Beh}) that at $n=p=2$ an analog of \eqref{eq:Behrens-fiber-seq} should be true and that there is a decomposition of $L_{K(2)}S^0$ in terms of spectra built from $E_2^{h\tmfg}, E_2^{hC_6}$ and $E_2^{hC_4}$ (see more on these subgroups of $\GG_2$ at $p=2$ in Section \ref{section:Background}) and their Spanier--Whitehead duals. The first step in proving such a result is the identification of the Spanier--Whitehead duals of the relevant spectra.    
 The main result of this paper is the following 2-primary statement which is analogous to the 3-primary case, perhaps hinting at common underlying structures.
\begin{theorem}\label{thm:theorem-one}
  Let $p=2$  and let $F$ be a finite
   subgroup of the Morava stabilizer group $\GG_2.$ Then there is a $K(2)$-local equivalence \[DE_2^{hF}\simeq \Sigma^{44}E_2^{hF}.\]
\end{theorem}

We prove this using the  short resolution of a spectrum closely related to $L_{K(2)}S^0$ at $p=2$ constructed in \cite{BG}. We use the associated spectral sequence to identify certain non-zero classes in $\pi_*DE_2^{hF}$, whose existence then forces $DE_2^{hF}\simeq\Sigma^{44}E_2^{hF}.$ 

\subsection*{Acknowledgments} I would like to thank Mark Behrens and Paul Goerss for inspiring conversations and Hans-Werner Henn for his comments on the early versions of this paper. 
I am very  grateful to the anonymous referee for a very thorough reading of several versions of this paper, for pointing out a mistake in the
first version, and for making numerous invaluable comments and suggestions for
improvement.
This material is based upon work supported by the National Science Foundation under Grant No.~ DMS-1440140, while the author was in residence at the Mathematical Sciences Research Institute in Berkeley, CA, during the Spring 2019 semester, Grant No.~DMS-1638352, while the author was a member at the Institute for Advanced Study in Princeton, NJ, and Grant No.~DMS-2005627.

 \section{Background}\label{section:Background}
 \subsection*{Morava \texorpdfstring{$E$}{E}-theory and the Morava stabilizer group}
 
Fix a prime number $p$.
The Morava stabilizer group $\SS_n$ is the group of automorphisms of the height $n$ Honda formal group law $H_n$ over $\mathbb{F}_{p^n}$. It is computed to be
$$\mathbb{S}_n=( W(\mathbb{F}_{p^n})\langle S\rangle/{(aS=Sa^{\sigma}, S^n=p)})^\times.
$$
Here, $W(\FF_{p^n})=\ZZ_p[\omega]$, where $\omega$ is a primitive  $p^{n}-1$-st root of unity, and $\sigma$ is the lift to $W(\FF_{p^n})$ of the Frobenius morphism $\sigma: \FF_{p^n} \xra{(-)^p} \FF_{p^n}.$
The lift to $\SS_n$ of the action of the Galois group $\Gal(\FF_{p^n}/{\FF_{p}})$ defines the (extended) Morava stabilizer group
\[
   \GG_n=\SS_n\rtimes \Gal(\FF_{p^n}/{\FF_p}).
\]
Goerss--Hopkins--Miller theory (see \cite{GH1}, \cite{RezkHMT}) produces a functor
\begin{align*}
   \mathbf{E}: \{\text{Formal group laws}\} &\to \{\mathcal{E}_{\infty}\text{-ring spectra} \}\\
   (k, \Gamma) &\mapsto \mathbf{E}(k, \Gamma)
\end{align*}
where $k$ is a perfect field of characteristic $p$ and $\Gamma$ is a  formal group law of finite height over $k.$
Morava $E$-theory $E_n$ at the prime $p$ is the value of this functor on $(\FF_{p^n}, H_n)$  where $H_n$ is, again, the Honda formal group law. 
%The spectrum $E_n=\mathbf{E}(\FF_{p^n}, H_n)$ is a complex oriented $\mathcal{E}_{\infty}$-ring spectrum on which $\GG_n$ acts via $\mathcal{E}_{\infty}$-maps. 
Its coefficients are 
$$
(E_n)_* \simeq W(\mathbb{F}_{p^n}) \llbracket u_1, \ldots u_{n-1} \rrbracket [u^{\pm 1}],
$$
where $|u_i|=0$ and $|u^{-1}|=2.$
By \cite{DH}, the homotopy fixed points of action of $\GG_n$ on $E_n$ recover the $K(n)$-local sphere
$
   L_{K(n)}S^0\simeq E_n^{h\GG_n}.
$
For any closed subgroup $H$ of $\GG_n$ we can form the homotopy fixed point spectrum $E_n^{hH}$ 
and by the results of \cite{DH}, there exists a fixed point spectral sequence
\[
   E_2^{*,*}=H^*_c(\GG_n, E_*E_n^{hH})\cong H^*(H, \pi_*E_n) \Longrightarrow \pi_*E_n^{hH}.
\]

Finite and other closed subgroups of $\GG_n$ play an important role in computations because they are often much easier to work with but still carry significant information. We will next  introduce several subgroups of interest to us at $n=2.$

We can write each element of $\GG_2$ as a pair 
$$(a+bS, \phi^e); \;\;\;\;\;\; a \in (W(\mathbb{F}_{p^2}))^{\times},\, b \in W(\mathbb{F}_{p^2}),\, e\in \{0,1\}$$ 
where $\phi$ is the Frobenius morphism,
and define the norm map by
\begin{equation*}
\begin{split}
\GG_2 &\xrightarrow N W(\FF_{p^2})^{\times} \rtimes \Gal(\FF_{p^n}/{\FF_p})
\\
(a+bS, \phi^e) &\mapsto (a\phi(a)-pb\phi(b), \phi^e).
\end{split}
\end{equation*}
 It is easy to check that the map $N$ takes values in $\ZZ_p^{\times}\times \Gal(\FF_{p^2}/{\FF_p}).$  
The group $\Gto$ is defined as the kernel of the reduced norm map
$$1 \rightarrow \Gto \rightarrow \GG_2 \xrightarrow{N}
\ZZ_p^{\times}\times \Gal(\FF_{p^2}/{\FF_p})\to
\mathbb{Z}_p^{\times}/{F}\simeq \mathbb{Z}_p \rightarrow 1,$$
where $F=C_2$ at $p=2$ and $F=C_{p-1}$ at $p\neq 2$ is the maximal finite subgroup of $\ZZ_p^{\times}.$
We also define $\SS_2^1:=\GG_2^1\cap \SS_2.$

\subsection*{Finite subgroups and short resolutions.}
The results of this paper concern the case $n=p=2.$ In order to explain our motivation for this project, we will first review some details about the $n=2$, $p=3$ case. We will write $E=E_2.$  
\subsubsection*{$\mathbf{p=3}$} 
It can be shown that $\SS_2$ contains a cyclic subgroup $C_3$ at $p=3;$ it is also easy to see that it contains a subgroup $C_8=\FF_9^{\times}$ generated by a primitive 8-th root of unity $\omega.$ Let $C_4$ be the subgroup generated by $\omega^2.$
The maximal finite subgroup $\tmfg$ of  $\GG_2$ has order $24$ and 
can be defined by the non-split group extension
\[
   1 \to C_3\rtimes C_4 \to \tmfg \to  \Gal(\FF_9/{\FF_3}) \to  1.
\]
 The semidihedral group 
$SD_{16}=(\FF_9)^\times\rtimes \Gal(\FF_9/{\FF_3})$ 
is the second finite subgroup of interest, note that it contains the dihedral group $D_8.$ 
See \cite[Section 1]{GHMR} for more details on the structure and generators of these subgroups.

The fixed point spectra $E^{hSD_{16}}$ and $\tmf$ are crucial for understanding the $K(2)$-local category at $p=3$ due to the existence of short resolutions of $L_{K(2)}S^0.$ Namely, in \cite{GHMR} the authors show that there exists a resolution of $L_{K(2)}S^0$
\begin{multline}\label{eq:ghmr-resolution}
L_{K(2)}S^0 \to \tmf \to \Sigma^8 \tmfsd\vee \tmf \to \\
\to \Sigma^8\tmfsd \vee \Sigma^{40}\tmfsd \to \Sigma^{40} \tmfsd \vee \Sigma^{48}\tmf \to \Sigma^{48}\tmf
\end{multline}   
Behrens \cite{Beh} used this resolution, and related calculations to show that there exists a fiber sequence
\[
   DQ(2) \to L_{K(2)}S^0 \to Q(2)
\]
where $Q(2)$ is built from $E^{hSD_{16}}$ and $\tmf.$ This explains the apparent self-duality of the resolution \eqref{eq:ghmr-resolution} and the presence of suspensions in it. Namely, the presence of $\Sigma^{48}\tmf$ in the resolution is related to the facts that $D\tmf\simeq \Sigma^{44}\tmf$ and that the resolution has length 4.  
\subsubsection*{$\mathbf{p=2}$}
It can be shown that $\SS_2$ at $p=2$ contains two elements of order four, $i$ and $j,$ which generate a subgroup $Q_8<\SS_2,$ on which $C_3=\FF_4^{\times}$ acts by permuting $i,j$ and $ij.$
There is one isomorphism class of  nonabelian, maximal, finite subgroups of $\SS_2$, given by the binary tetrahedral group $\tmfg=Q_8\rtimes \FF_4^{\times}$ (see \cite[Corollary 1.5]{Hewett}) and $G_{48}:=G_{24}\rtimes \Gal(\FF_4/{\FF_2}) \subseteq \GG_2$ (for details on these subgroups see, for example, \cite[Section 2]{HennCentralizer}). Note that the group $\tmfg$ at $p=2$ is not the same as the maximal finite subgroup of $\GG_2$ at $p=3,$ even though the usual notation is the same. We will also use subgroups $C_2=\{\pm 1\}$ and $C_6=\{\pm 1\}\times \FF_4^{\times}.$ 

The analog of \eqref{eq:ghmr-resolution} at $p=2$ is the following resolution of $\eS$ \cite{BG}
\begin{equation}\label{eq:BG-resolution}
  \eS \to \tmf \to E^{hC_6} \to \Sigma^{48}E^{hC_6} \to \Sigma^{48}\tmf.
\end{equation}
Both resolutions \eqref{eq:ghmr-resolution} and \eqref{eq:BG-resolution} have the property that all possible Toda brackets formed from the maps in the resolutions are zero. Hence they refine to towers of fibrations and give rise to tower spectral sequences. Namely, \eqref{eq:BG-resolution} refines to a tower of fibrations
	{\small{\begin{equation}\label{eq:my-tower}
	   	\xymatrix{
			\mathfrak{F}_3=\Sigma^{45}\tmf \ar[r] & E^{h\mathbb{S}_2^1} \ar[d] \\
\mathfrak{F}_2=\Sigma^{46}E^{hC_6}\ar[r] & X_2 \ar[d]\\
\mathfrak{F}_1=\Sigma^{-1}E^{hC_6}\ar[r] & X_1 \ar[d]\\
\mathfrak{F}_0=\tmf \ar[r]^-{\simeq} &\tmf.}
	\end{equation}}}
 and we have a tower spectral sequence
 \[
    E_1^{t,s}=\pi_t\mathfrak{F}_s \Longrightarrow \pi_{t-s}\eS.
 \]
 
We can map $\tmf$ (or any other spectrum) into the tower \eqref{eq:my-tower} and arrive at a spectral sequence computing the homotopy of the function spectrum
 \begin{equation}\label{eq:towerSS-function-spectrum}
        E_1^{t,s}=\pi_tF(\tmf,\mathfrak{F}_s) \Longrightarrow \pi_{t-s}F(\tmf, \eS).
 \end{equation} 

\section{Action of \texorpdfstring{$\GG_n$}{Gn} on function spectra}\label{section:actions}

In this section we work at an arbitrary chromatic height $n$ and  write $\GG=\GG_n$ and $E=E_n$ in order to simplify the notation. 
For an element $g\in \GG$ and $\alpha \in \pi_nF(E,E)$ let $t_g$ and $s_g$ denote the actions on the target and source:
\begin{align}
t_g(\alpha) &= (\Sigma^n E\xra{\alpha} E \xra{g}E) =g \circ\alpha \label{eq:tg}\\
s_g(\alpha) &= (\Sigma^n E \xra{g} \Sigma^n E\xra{\alpha} E) =\alpha \circ g. \label{eq:sg}
\end{align}
Assuming that $\GG$ acts on $E$ on the left, we can see that $t_g$ is a left action of $\GG$ on $F(E,E)$, and $s_g$ is a right action of $\GG$ on $F(E,E)$, hence we now have a left action of $\GGop$ on $F(E,E)$.

There exist at least two ways to define a $\GGop$ action on $E_*\llbracket\GG\rrbracket$ so that 
there was a $\GGop$ equivariant isomorphism $E_*\llbracket\GG\rrbracket \cong \pi_*F(E,E)$, with the action on $\pi_*F(E,E)$ as in \eqref{eq:tg} and \eqref{eq:sg}. 
Non-equivariant and $\GG$-equivariant versions of this isomorphism are also discussed in \cite{DH}, \cite{StricklandGrossHopkins}, \cite{BehrensDavis}, \cite{GHMR},
and \cite{HoveyEStar}.

\subsection{First isomorphism}
In this paper we will use the isomorphism discussed below in Section \ref{sec:second-equivalence}, but another equivariant isomorphism is used more often, and we would like to write down some details about it first.

Let $(a \in E_*,\gamma\in \GG)$ be an element of $E_*\llbracket\GG\rrbracket$. For $g\in \GG$, consider the two actions on $E_*\llbracket\GG\rrbracket$: the right action given by
\begin{equation}
  r_g(a,\gamma)=(a, \gamma g),
\end{equation}
and the left action given by ($g.a$ denotes the action of $g$ on $a\in E_*$)
\begin{equation}
  l_g(a, \gamma)=(g.a, g\gamma).
\end{equation}
\begin{thm}\label{thm:first-equivalence}
  There exists a $\GGop$ equivariant isomorphism of $E_*$-algebras
   \[\phi: E_*\llbracket\GG\rrbracket \to \pi_*F(E,E),\]
such that
\[
\phi(r_g(a,\gamma))=s_g(\phi(a,\gamma))\]
\[
\phi(l_g(a, \gamma))=t_g(\phi(a,\gamma)).\]
\end{thm}
\begin{proof}
The non-equivariant version of this statement is proved in \cite[Theorem 5.5]{HoveyEStar}, and the equivariant version is discussed in \cite{GHMR}. Hovey defines the map $\phi$ as
\begin{align*}
   \phi: E_*\llbracket\GG\rrbracket &\to \pi_*F(E,E)\\
   ( S^n\xra{a} E,\gamma) &\mapsto (\Sigma^n E \xra{\gamma}\Sigma^nE \xra{a} E\wedge E\xra{\mu} E)
\end{align*}
and proves that it is an isomorphism of $E_*$-algebras. In order to prove the equivariant statement, 
let $g\in \GG$. Then we have
\begin{align*}
s_g(\phi(a,\gamma))&=s_g(\Sigma^n E \xra{\gamma}\Sigma^nE \xra{a} E\wedge E\xra{\mu} E)\\
 &= (\Sigma^n E \xra{g}\Sigma^n E \xra{\gamma}\Sigma^nE \xra{a} E\wedge E\xra{\mu} E)\\
&=(\Sigma^n E \xra{ \gamma g}\Sigma^nE \xra{a} E\wedge E\xra{\mu} E)=\phi(a, \gamma g)=\phi(r_g(a,\gamma)).
\qedhere
\end{align*}
For the other two actions we have:
\begin{align*}
t_g(\phi(a,\gamma))&=t_g(\Sigma^n E \xra{\gamma}\Sigma^nE \xra{a} E\wedge E\xra{\mu} E)\\
 &= (\Sigma^n E \xra{\gamma}\Sigma^nE \xra{a} E\wedge E\xra{\mu} E \xra{g} E)\\
&=(\Sigma^n E \xra{g \gamma}\Sigma^nE \xra{g.a} E\wedge E\xra{\mu} E)=\phi(g.a, g\gamma)=\phi(l_g(a,\gamma))
\end{align*}
\end{proof}
\subsection{Second isomorphism}\label{sec:second-equivalence}
In this section we would like to consider a different action of $\GGop$ on $E_*\llbracket\GG\rrbracket$. Namely, let the left action $L_g$ of $\GG$ and the right action $R_g$ of $\GG$ be as follows:
\begin{align*}
L_g(a,\gamma)&=(a, g\gamma)\\
R_g(a,\gamma)&=(g^{-1}.a,\gamma g).
\end{align*}
\begin{thm}\label{thm:second-equivalence}
There exists a $\GGop$ equivariant isomorphism of $E_*$-algebras
   \[\psi: E_*\llbracket\GG\rrbracket \to \pi_*F(E,E),\]
such that 
\[
\psi(L_g(a,\gamma))=t_g(\psi(a,\gamma)).
\]
\[
\psi(R_g(a, \gamma))=s_g(\psi(a,\gamma))
\]
\end{thm}
\begin{proof}
We define the map $\psi$ as (following \cite[p.1029]{StricklandGrossHopkins})
\begin{align*}
   \psi: E_*\llbracket\GG\rrbracket &\to \pi_*F(E,E)\\
   ( S^n\xra{a} E,\gamma) &\mapsto (\Sigma^n E\xra{a}E\wedge E \xra{\mu} E \xra{\gamma}E)
\end{align*}
Since
$
   \psi(a,\gamma)=\phi(\gamma.a, \gamma)
$,
$\psi$ is an isomorphism of $E_*$-algebras. 
We will check the equivariant part of the statement, just as in Theorem \ref{thm:first-equivalence}:
\begin{align*}
t_g(\psi(a,\gamma))&=t_g(\Sigma^nE \xra{a} E\wedge E\xra{\mu} E \xra{\gamma} E)\\
 &= (\Sigma^nE \xra{a} E\wedge E\xra{\mu} E \xra{\gamma} E \xra{g} E)\\
&=(\Sigma^nE \xra{a} E\wedge E\xra{\mu} E \xra{g \gamma} E)=\psi(a, g\gamma) = \psi(L_g(a, \gamma))
\end{align*}
and
\begin{align*}
s_g(\psi(a,\gamma))&=s_g(\Sigma^nE \xra{a} E\wedge E\xra{\mu} E \xra{\gamma} E)\\
 &= (\Sigma^n E \xra{g}\Sigma^nE \xra{a} E\wedge E\xra{\mu} E \xra{\gamma} E)\\
&=(\Sigma^nE \xra{g^{-1}.a} E\wedge E\xra{\mu} E \xra{\gamma g} E)=\psi(g^{-1}.a, \gamma g) = \psi(R_g(a, \gamma)).
\end{align*}
The hardest part of this is the third equality which can be visualized as follows (we are using that $\gamma$ and $\gamma g$ act by ring maps)
\[
\gamma .(a(g.x))=(\gamma .a) (\gamma.(g.x))=((\gamma g g^{-1}).a )((\gamma g).x)=\gamma g((g^{-1}.a)x).
\qedhere
\]
\end{proof}
Now assume we would like to understand $\pi_*F(E^{hK}, E^{hH})$ for $H, K\subseteq \GG$.  The isomorphism $\phi$ is often better suited to this task if we first take the fixed points with respect to $K$, and then the fixed points with respect to the diagonal action of $H$ on $\pi_*F( E^{hK}, E)$. This is the approach used in \cite{GHMR} and \cite{BehrensDavis}. But if we wanted to take the fixed points in the other order, it might be more advantageous to use the isomorphism $\psi$. This is the approach we will take in this paper, inspired by \cite{StricklandGrossHopkins}, where $\psi$ was used to compute $DE$.

\section{Spanier--Whitehead dual of \texorpdfstring{$\tmf$}{tmf}: first steps}
For the rest of the paper we work in the $K(2)$-local category at $p=2$ and write $E=E_2.$
We begin with a recollection of computations from \cite[Prop. 16]{StricklandGrossHopkins}, where it was proved that there is an equivalence of $E$-modules (for $n=2$ and any prime)
$$
F(E, E^{h\mathbb{G}_2})\simeq \Sigma^{-4}E,
$$ 
inducing a $\GG_2$-equivariant isomorphism on homotopy groups. 
This can be shown by computing $H^*(\mathbb{G}_2, E^*E)$ and using the Devinatz--Hopkins spectral sequence \cite{DH}
\begin{equation}\label{eq:DevinatzHopkinsSS}
   E_2^{s,t}=H^s(G, \pi_tF(X,Z)) \Longrightarrow \pi_{t-s}F(X, Z^{hG}).
\end{equation}
We will emulate Strickland's analysis for the group $\GG_2^1 \subset \mathbb{G}_2$ instead of $\GG_2$. One of the key facts we need is that $\GG_2$ contains an open Poincar\'e duality subgroup of dimension 4.
A profinite $p$-group $G$ is called a Poincar\'e duality group of dimension $k$ if $G$ has cohomological dimension $k$ and 
   \[
      H^n_c(G, \ZZ_p \llbracket G \rrbracket )=
      \begin{cases}
         \ZZ_p &n=k\\
         0 &n\neq k,
      \end{cases}
   \]
  where the action used to define group cohomology is the  natural left action of $G$ on the topological ring $\ZZ_p\llbracket G\rrbracket$.
   A group $H$ is called a \emph{virtual} Poincar\'e duality group if it possesses a finite-index subgroup $G$ which is a Poincar\'e duality group.
The groups $\GG_2$ and $\Gto$ are not Poincar\'e duality groups, but they are virtual Poincar\'e duality groups. 
\begin{lemma}\label{lemma:poincare}
   We have 
   \[
     H^n_c(\GG_2^1, \ZZ_2 \llbracket \Gto \rrbracket )=
   \begin{cases}
   \ZZ_2 &n=3\\
   0 &n\neq 3.
   \end{cases} 
   \]
\end{lemma}
\begin{proof}
   The group $\GG_2$ contains an open Poincar\'e duality subgroup $K$ of dimension 4 and $\GG_2^1$ contains an open Poincar\'e duality subgroup $K^1=K\cap \Gto$ of dimension 3. For details on generators of $K$ and $K^1$ and their properties, see \cite[Section 2]{BE}. 
 Open subgroups of a profinite group are precisely those closed subgroups which have finite index.
   In fact, $K^1$ fits into a short exact sequence
   \[
      1 \to K^1 \to \Sto \to \tmfg \to 1.
   \]
   Then we have an isomorphism supplied by the Shapiro's lemma
   \[
      H^*_c(\Gto, \ZZ_2 \llbracket \Gto \rrbracket )\cong H^*_c(K^1, \ZZ_2 \llbracket K^1 \rrbracket ).
   \qedhere\]
\end{proof}

\begin{lemma}\label{lem:Lemma5}
\begin{enumerate}
\item There exists a $\GG_2$-equivariant isomorphism
\[
 \pi_*F(E,  E^{h\GG_2^1}) \cong \pi_*\Sigma^{-3}E \br{\G/{\GG_2^1}},
\]
where the action of $\GG_2$ on the left hand side is on the source in an element of $\pi_*F(E, \eG)$ and the action of $g\in \GG_2$ on the right hand side is given by $g.(a, \gamma \GG_2^1)=(g^{-1}a, (\gamma g)\GG_2^1)$ for $a\in E_*$ and a coset $\gamma\GG_2^1$. These are right actions of $\GG_2$.
\item The isomorphism  of (1) is also equivariant with respect to the left action of the group $\GG_2/{\GG_2^1}\cong \ZZ_2$, which is the residual action on the target in the function spectrum, and the natural action on $\GG_2/{\GG_2^1}$ on the right hand side.
\end{enumerate}
\end{lemma}
\begin{proof}
We will use the spectral sequence \eqref{eq:DevinatzHopkinsSS}
and we 
need to compute 
\[
E_2^{*,*}\cong H^*_c(\Gto, \pi_*F(E, E)).
\]
Our starting point is the isomorphism of Theorem~\ref{thm:second-equivalence}:
$
 \pi_*F(E,E) \cong E_* \llbracket \GG_2 \rrbracket  
$,
under which the action on the target in  $\pi_*F(E,E)$ corresponds to the action which we called $L_g$ in Section \ref{sec:second-equivalence}, namely 
$g.(a,\gamma)=(a, g\gamma).$

There is an isomorphism of $\GG_2^1$-modules 
\begin{equation}\label{eq:first-decomp}
E_*\llbracket\mathbb{G}_2\rrbracket \cong E_*\llbracket\mathbb{G}_2^1\rrbracket \widehat{\otimes}_{E_*}E_*\llbracket\mathbb{G}_2/{\GG_2^1}\rrbracket
\end{equation}
where $E_*\llbracket\mathbb{G}_2/{\GG_2^1}\rrbracket$ has trivial action. 
This gives an isomorphism 
\[
H^*_c(\GG_2^1, E_*\llbracket\mathbb{G}_2\rrbracket)\cong H^*_c(\GG_2^1, E_*\llbracket\GG_2^1\rrbracket)\widehat{\otimes}_{E_*} E_*\llbracket\mathbb{G}_2/{\GG_2^1}\rrbracket
\]
and we are reduced to computing $H^*_c(\GG_2^1, E_*\llbracket\GG_2^1\rrbracket)$. But since $\GG_2^1$ acts trivially on $E_*$ (by definition of $L_g$), we have an isomorphism of $\mathbb{G}_2^1$-modules 
$$
E_* \llbracket \GG_2^1\rrbracket\cong  \ZZ_2\llbracket\mathbb{G}_2^1\rrbracket \widehat{\otimes}_{\ZZ_2} E_*$$
and $H^*_c(\GG_2^1, E_*\llbracket \GG_2^1\rrbracket)=H^*_c(\GG_2^1, \ZZ_2\llbracket \GG_2^1\rrbracket)\widehat{\otimes}_{\ZZ_2} E_*$.

Hence we have 
\[
 H^n_c(\Gto, E^tE)=
  \begin{cases}
   E_t\llbracket \GG_2/{\GG_2^1}\rrbracket, &n=3\\
   0, &n\neq 3.
   \end{cases}
   \]

Substituting these results into the spectral sequence \eqref{eq:DevinatzHopkinsSS}
 we see that it cannot have any differentials or extensions due to sparseness. Hence it collapses and we have an isomorphism of homotopy groups
$$
\pi_* F(E, E^{h\Gto})\cong \pi_* \Sigma^{-3}E\br{\G/{\Gto}}.
$$
Theorem \ref{thm:second-equivalence} implies that 
 this isomorphism
is equivariant with respect to the right action of $\GG_2$: on the source in the function spectrum and induced from $R_g$ on $\pi_*\Sigma^{-3}E \llbracket \G/{\GG_2^1} \rrbracket$. 

To prove the second statement, we keep track of the action of $\GG_2/{\GG_2^1}$ at each step, starting with \eqref{eq:first-decomp}, where the action is the natural action of the group on the group ring $E_*\llbracket \G/{\GG_2^1} \rrbracket$. This action extends to the action on group cohomology, and on the $E_2=E_{\infty}$ page of the spectral sequence.
\end{proof}
In this paper we want to identify the function spectrum $F(\dtmf, E^{h\GG_2})$ and we will do that by first understanding $F(\tmf, \eG).$ The next lemma allows us to think about the latter function spectrum as a homotopy fixed point spectrum.
\begin{lemma}\label{lem:Tate_vanishes}
Let $H$ be any finite subgroup of $\GG_2.$ Then we have an equivalence
$$
 F(E^{hH}, E^{h\mathbb{G}_2^1})\simeq F(E, E^{h\mathbb{G}_2^1})^{hH}.
$$
\end{lemma}

\begin{proof}
Since the $K(2)$-localization of the Tate spectrum $E^{tH}$ vanishes, the homotopy fixed point spectrum $E^{hH}$ and the homotopy orbit spectrum $E_{hH}$ are equivalent. The lemma then follows from the equivalences
\[
F(E^{hH}, E^{h\Gto})\simeq F(E_{hH}, E^{h\Gto})\simeq F(E, E^{h\Gto})^{hH}.\qedhere
\]
\end{proof}

The homotopy of $F(E, E^{h\GG_2^1})^{h\tmfg}$ can now be computed using the fixed point spectral sequence 
\[
   E_2^{s,t}=H^s(\tmfg, \pi_tF(E, E^{h\mathbb{G}_2^1})) \Longrightarrow \pi_{t-s}F(E, \eG)^{h\tmfg}.
\]
We will later show that we actually only need to know very little information about some permanent cycles in order to compute this spectral sequence completely.
\subsection*{Fixed point spectral sequence for $\pi_*\tmf$}

Here we will recall some basic facts about the homotopy fixed point spectral sequence
\begin{equation}\label{eq:tmf-ss}
   E_2^{s,t}=H^s(\tmfg, \pi_tE) \Longrightarrow \pi_{t-s}\tmf.
\end{equation}
For more details see \cite[Theorem 18.2]{Rezk512}, \cite{Tilman} or \cite[Section 2.3]{BG}. 

There is an isomorphism 
$$
H^0(\tmfg, E_*)\cong W(\FF_4) \llbracket j \rrbracket [c_4, c_6,\Delta^{\pm 1}]/{(c_4^3-c_6^2=(12)^3\Delta, \Delta j=c_4^3)}
$$
and 
\[
   H^*(\tmfg, E_*)\cong H^0(\tmfg, E_*)[\eta, \nu, \mu, \epsilon, \kappa, \kappabar]/R
\]
where $R$ is the ideal generated by the following relations:
\begin{gather*}
2\eta=2\mu=2\epsilon=2\kappa=4\nu=8\kappabar=0;\\
\eta^2\kappa=\eta\nu= 2\nu^2=\nu^4=0;\\
\eta\epsilon=\nu^3, \,\,\nu\epsilon=\epsilon^2=0,\,\, \nu^2\kappa=4\kappabar, \,\, \epsilon\kappa=\kappa^2=0;\\
\mu\nu=c_4\nu=c_6\nu=0,\,\, \mu\epsilon=c_4\epsilon=c_6\epsilon=0,\,\, \mu\kappa=c_4\kappa=c_6\kappa=0;\\
\mu^2=\eta^2c_4,\,\, \mu c_4=\eta c_6,\,\, \mu c_6=\eta c_4^2,\,\,
c_4\kappabar=\eta^4\Delta,\,\, c_6\kappabar=\eta^3\mu\Delta.
\end{gather*}
We include for reference the chart for the $E_2$ page of the spectral sequence \eqref{eq:tmf-ss} in 
Figure \ref{figure:tmf-shifted}. This is Figure 3 from \cite{BG}. The notation in Figure \ref{figure:tmf-shifted} is as follows: $\Box=W(\FF_4) \llbracket j \rrbracket $, $\Circle=\FF_4 \llbracket j \rrbracket $, $\otimes = W(\FF_4) \llbracket j \rrbracket /(8,2j)$ generated by a class of the form $\Delta^i\kappabar^j,$ a bullet denotes a class of order 2 and a circled bullet is a class of order 4. The solid lines of slope 1 indicate multiplication by $\eta$ and lines  of slope 1/3 indicate multiplication by $\nu.$ A dashed line indicates
that $x\eta=jy$, where $x$ and $y$ are generators in the appropriate bidegrees. 
\begin{figure}[h]
 \includegraphics[page=1, width=0.8\textwidth]{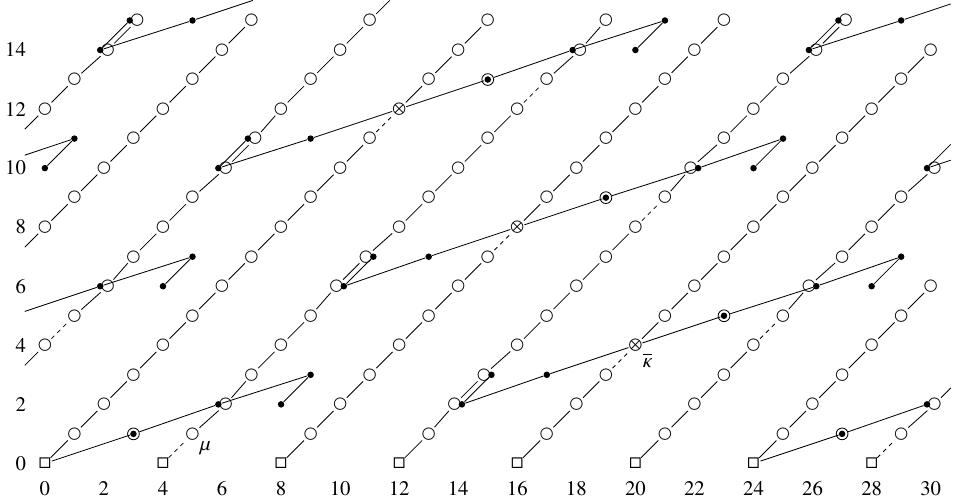}
\caption{The $E_2$ page of the spectral sequence \eqref{eq:tmf-ss}. 
The horizontal axis is $t-s$ and the vertical axis is $s.$
}
\label{figure:tmf-shifted}
 \end{figure}
The $E_2$ page is 24-periodic with periodicity generator  $\Delta\in H^0(\tmfg, E_{24}).$ This algebraic periodicity does not extend to topological periodicity since the spectral sequence \eqref{eq:tmf-ss} has differentials on the powers of $\Delta$ given by
\begin{align*}
   d_5(\Delta)&=\kappabar \nu\\
d_7(\Delta^4)&=\Delta^3\kappabar \eta^3.
\end{align*}
The differentials are linear with respect to $j$ and the generators in positive $s$ degrees.
The spectrum $\tmf$ is 192-periodic with periodicity generator detected by the permanent cycle $\Delta^8\in H^0(\tmfg, E_{192}).$ 

\subsection*{The function spectrum $F(\tmf, \eG)$}
The spectrum $F(E, \eG)$ has two group actions as given in Lemma \ref{lem:Lemma5}: the action of $\GG_2$ on the source and the residual action of $\GG_2/{\GG_2^1}$ on the target. The action on the source is used in computing the $E_2$ page of the fixed  point spectral sequence
   \[
     E_2^{*,*}= H^*(\tmfg, \pi_*F(E, \eG)) \Longrightarrow \pi_*F(\tmf, \eG).
   \]
    Since $\tmfg$ 
is  a subgroup of $\Gto,$ it acts trivially on $\GG_2/{\Gto}$ and we have 
\[
   E_2^{*,*} \cong H^*(\tmfg, \pi_*\Sigma^{-3}E) \llbracket \GG_2/{\Gto} \rrbracket .
\]

\begin{prop}\label{prop:Delta-is-perm-cycle-for-Dtmf}
Consider the
 fixed point spectral sequence
\begin{equation}\label{eq:main-spectral-sequence}
    E_2^{s,t}=H^s(G_{24}, \pi_t\Sigma^{-3}E) \llbracket \GG_2/{\Gto} \rrbracket  \Longrightarrow \pi_{t-s} F(\tmf,\eG).
\end{equation}
Assume that for some $k\in [0,7]$ (and all $n$)
\[
\Delta^{8n+k} \in E_2^{0,-3+24k+192n}=H^0(\tmfg, E_{24k+192n})\llbracket\GG_2/{\GG_2^1}\rrbracket
\]
is a permanent cycle. 
Then
\[
F(\tmf, \eG)\simeq \Sigma^{-3+24k}\tmf\llbracket\GG_2/{\GG_2^1}\rrbracket.
\]
\end{prop}

\begin{proof}

Spectral sequence \eqref{eq:main-spectral-sequence} is a module over the standard homotopy fixed point spectral sequence  \eqref{eq:tmf-ss}.
Using this module structure we see that if $\Delta^{k+8n}$ is a permanent cycle, then any element 
\[a \in H^0(\tmfg, E_{24k+192n})\subset E_2^{0,-3+24k+192n}\]
is a permanent cycle as well.

Now we note that there is the residual action of the group $\GG_2/{\GG_2^1}$ on the target in the function spectra $F(E, \eG)$ and $F(\tmf, \eG)$, and on the homotopy fixed point spectral sequence \eqref{eq:main-spectral-sequence}, hence the differentials commute with this action. To be more explicit, consider a coset $[g]\in \GG_2/{\GG_2^1}.$
By Lemma \ref{lem:Lemma5}, part (2), the action of $[g]$ on $\pi_* F(E, \eG)=\pi_*\Sigma^{-3}E\llbracket\GG_2/{\GG_2^1}\rrbracket$ is  trivial on $\pi_*\Sigma^{-3}E$ and natural on $\GG_2/{\GG_2^1}$: $[g][h]=[gh]$. 
Hence for
 \[a[h] \in H^0(\tmfg, E_{24k+192n})\llbracket\GG_2/{\GG_2^1}\rrbracket= E_2^{0,-3+24k+192n},\]
  we have $[g].(a[h])=a[gh]$, where on the left $[g].(a[h])$ denotes the action of $[g]\in \GG_2/{\GG_2^1}$ on $a[h]$, and on the right $a[gh]$ is an element of $E_*\llbracket\GG_2/{\GG_2^1}\rrbracket.$ This allows us to write any $a[h]$ as the result of the action $a[h]=h.a[1]$ and we have
\[
d_r(a[h])=d_r([h].(a[1]))=[h].d_r(a)
\]
%Take a coset $[h]\in \GG_2/{\GG_2^1}$ and consider it as an element $1\wedge [h] \in \pi_*F(\eG, \eG)$, which acts on   $F(E, \eG)$. Take $a \in \pi_0 F(E, \eG)$ which is invariant under $\tmfg,$ i.e. it lifts to a cocycle $a\in H^0(\tmfg, \pi_0F(E,\eG).$  Then $a[h] \in H^0(\tmfg, \pi_0F(E, \eG)).$ And we can compute $d_r^{\eqref{eq:eG-fixed-point}}(a[h])=d_r^{F(E, \eG)}(a)d_r^{F(\eG, \eG)}([h])=0.$

This shows that the differentials in this spectral sequence are linear with respect to elements in $\GG_2/{\GG_2^1}$  and we have 
\[
\pi_*F(\tmf, \eG)\cong \pi_*\Sigma^{-3+24k}\tmf\llbracket\GG_2/{\GG_2^1}\rrbracket.   
\]
Now note that (see \cite[Lemma 2.3.5]{Beh})
	\[
	\tmf\llbracket\GG_2/{\GG_2^1}\rrbracket\simeq \tmf\wedge S\llbracket\GG_2/{\GG_2^1}\rrbracket.
	\]
Then the composition
\begin{multline*}
F(\tmf, \eG)\wedge \tmf\wedge S\llbracket\GG_2/{\GG_2^1}\rrbracket \xra{\mu}\\
 \xra{\mu} F(\tmf, \eG)\wedge S\llbracket\GG_2/{\GG_2^1}\rrbracket \xra{\xi} F(\tmf, \eG),
\end{multline*}
where the map $\xi: \eG\wedge S\llbracket\GG_2/{\GG_2^1}\rrbracket\to \eG$ is the action map of \cite[Cor. 2.3.4]{Beh} and $\mu:\tmf\wedge F(\tmf, \eG)\to F(\tmf, \eG)$ is the module structure map,
gives $F(\tmf, \eG)$ the structure of a module over $\tmf\llbracket\GG_2/{\GG_2^1}\rrbracket.$
Using this module structure and  the map 
\[
\widetilde{\Delta}^k: S^{-3+24k} \to F(\tmf, \eG),
\]
detected by the permanent cycle $\Delta^k$, we get the required equivalence of spectra.
\end{proof}

The next lemma shows that we can relax the assumptions of Proposition \ref{prop:Delta-is-perm-cycle-for-Dtmf}.
\begin{lemma}\label{lemma:Delta-f(j)-enough}
   Consider the homotopy fixed point spectral sequence \eqref{eq:main-spectral-sequence}
   \[E_2^{s,t}=H^s(G_{24}, E_{t+3})\llbracket\GG_2/{\GG_2^1}\rrbracket \Longrightarrow \pi_{t-s+3} F(\tmf,\eG).\]
   Assume that $\Delta^{k}f(j) \in E_2^{0,-3+24k+192n}$ is a permanent 
cycle, where 
$f(j)$ is a power series in $j$ such that 
$f(0)\neq 0$ mod $(2).$
Then $\Delta^k $ is a permanent cycle.  
\end{lemma}
\begin{proof}
The spectral sequence \eqref{eq:main-spectral-sequence} is a module over the standard homotopy fixed point spectral sequence for $\pi_*\tmf$ and the differentials in the latter spectral sequence are $j$-linear. Hence we have
\[ 
   d_r(\Delta^k f(j))=f(j)d_r(\Delta^k)=0.
\]
And the condition $f(0)\neq 0$ mod $(2)$ ensures that $f(j)$
is invertible in the target of $d_r$.
\end{proof}
In the next section we will show that $\Delta^{2+8n}f(j)$ for $f(j)$ as in Lemma \ref{lemma:Delta-f(j)-enough} is a permanent cycle in spectral sequence \eqref{eq:main-spectral-sequence}. Then we will use Proposition \ref{prop:Delta-is-perm-cycle-for-Dtmf} and Lemma \ref{lemma:Delta-f(j)-enough} to deduce $F(\tmf, \eG) \simeq \Sigma^{45}\tmf \llbracket \GG_2/{\Gto} \rrbracket .$

\section{Homotopy groups computation}\label{sec:htpy-computation}
We begin by examining the   homotopy fixed point 
spectral sequence for $\pi_*\tmf$ (\cite[Fig. 4]{BG} or \cite[p. 32]{Tilman})
   and making the following observation. 
\begin{lemma}\label{lemma:f_4}
For any $n$ there is an isomorphism
\[
   \pi_{45+192n}\tmf \cong \FF_4.
\]
If $a_n\in \pi_{45+192n}\tmf$ is a generator
then $a_n\kappabar\eta\neq 0.$ 
Furthermore, $a_n$ is detected by the class
$\Delta^{1+8n}\kappabar\eta \in H^5(\tmfg, \pi_{50+192n}E).$
\end{lemma}
The next result follows from \cite[Prop. 2.6]{GHMR} and details can be found in \cite[p.925]{BG}. Let $H_1$ be a closed subgroup and $H_2$ a finite subgroup of $\GG_2$, and let $H_1=\cap_i U_i$ for a decreasing sequence of open subgroups $U_i.$ Then we have an isomorphism 
\begin{equation}\label{eq:decomposition}
   \pi_* F(E^{hH_1}, E^{hH_2})\cong  \lim\prod_{H_2\backslash \GG_2/{U_i}}\pi_* E^{hH_{x,i}},
\end{equation}
where $H_{x,i}=H_2\cap xU_ix^{-1}\subseteq H_2$ is the isotropy subgroup of the coset $xU_i.$ 

In order to use this decomposition we will need some information about $\pi_*E^{hH}$ for various subgroups $H\subseteq\tmfg.$ What we need is collected in the lemma below.
\begin{lemma}\label{lemma:pi-minus-one}
   \begin{enumerate}
      \item For any $H\subseteq\tmfg$ such that the central $C_2=\{\pm 1\}$ is contained in $H$ we have  \[\pi_{-1}E^{hH}=0.\]
      \item 
 For $F=C_2 \subset \tmfg$ and $C_6\subset \tmfg$ 
\[
   \pi_{45+192n}E^{hF}=\pi_{46+192n}E^{hF}=0.
   \]
   \end{enumerate}
\end{lemma}
\begin{proof}
   The subgroups of $\tmfg$ which contain the central $C_2$ are $C_2, C_4, C_6, Q_8$ and $\tmfg.$ The homotopy groups of $E^{hC_2}, E^{hC_4}$ and $E^{hC_6}$ can be read off of Prop. 2.8, Prop. 2.9 and Prop. 2.12 in \cite{BG}. For $E^{hQ_8}$ we note that there is an equivalence (\cite[p.28]{HennCentralizer})
   \[
      E^{hQ_8} \simeq \tmf \vee \Sigma^{64} \tmf \vee \Sigma^{128}\tmf
   \]
and $\pi_{-1}\tmf=\pi_{63}\tmf=\pi_{127}\tmf=0.$
\end{proof}

\begin{lemma}\label{lemma:map-p-prime}
There exists (for each $n$) a surjective map
\[
   \pi_{45+192n}F(\tmf,\eS) \xra{p'} \FF_4
\]
such that any element \[{f_n}\in \pi_{45+192n}F(\tmf,\eS),\] for which  $p'({f_n})\neq 0,$ has Adams--Novikov filtration at most 5 and has the property $f_n\kappabar\eta\neq 0.$
\end{lemma}
\begin{proof}
 We compute with the tower spectral sequence \eqref{eq:towerSS-function-spectrum}
  \begin{equation*}
        E_1^{s,t}=\pi_tF(\tmf,\mathfrak{F}_s) \Longrightarrow \pi_{t-s}F(\tmf, \eS).
 \end{equation*} 
 where we examine the fate of 
 \[
 E_1^{0,{45+192n}}\cong \pi_{45+192n}F(\tmf, \tmf).
 \]
The three potential differentials supported by $E_1^{0,{45+192n}}$ land in 
\begin{align*}
\pi_{45+192n}F(\tmf, E^{hC_6})&=\pi_{46+192n} F(\tmf, E^{hC_6})=\\
&=\pi_{47+192n} F(\tmf, \Sigma^{48}\tmf)=0,  
\end{align*}
all of which are zero groups by  \eqref{eq:decomposition} and Lemma \ref{lemma:pi-minus-one}.
\noindent
Then 
$E_1^{0, {45+192n}}\cong E_{\infty}^{0, {45+192n}}$,
and the projection $p$ from the top to the bottom of the duality tower \eqref{eq:my-tower}
\[
   \pi_{45+192n}F(\tmf, \eS) \xra{p} \pi_{45+192n}F(\tmf,\tmf)
\]
is surjective.
Composing it with the unit map $\iota$  of the ring spectrum $\tmf$ we have a surjective map $p'$
\[
   p': \pi_{45+192n}F(\tmf, \eS) \xra{p} \pi_{45+192n}F(\tmf,\tmf)\xra{\iota} \pi_{45+192n}\tmf\cong \FF_4.
\]
The rest follows from Lemma \ref{lemma:f_4}.
\end{proof}

\begin{corollary}\label{corollary:important}
In the spectral sequence \eqref{eq:main-spectral-sequence}
\begin{equation*}
E_2^{s,t}=H^s(G_{24}, \pi_t\Sigma^{-3}E\llbracket\GG_2/{\GG_2^1}\rrbracket )\Longrightarrow \pi_{t-s} F(\tmf,\eG).   
\end{equation*}
there exists (for each $n$) a permanent cycle
\[
   \Delta^{2+8n} g_n(j) \in E_{2}^{0,-3+48+192n}=H^0(\tmfg, E_{-3+48+192n}) \llbracket \GG_2/{\GG_2^1} \rrbracket 
\]
such that $g_n(0)\neq 0$ mod $(2).$
\end{corollary}

\begin{proof}
We use the fact that
$F(\tmf, \eS)=F(\tmf, \eG)\wedge \Gal(\FF_4/{\FF_2})_+$ (Lemma 1.37 in \cite{BG}).
Then the map 
	\[
	p': \pi_{45+192n}F(\tmf, \eS) \to \FF_4
	\]
from Lemma \ref{lemma:map-p-prime} restricts to a surjective map
\[
r:\pi_{45+192n}F(\tmf, \eG) \to \FF_2    
\]
where any $y_n\in \pi_{45+192n}F(\tmf, \eG)$ such that $r(y_n)\neq 0$ is detected in the spectral sequence \eqref{eq:main-spectral-sequence} by a permanent cycle in Adams-Novikov filtration at most 5. 

Now we analyze the spectral sequence 
\[
E_2^{s,t}=H^s(G_{24}, \pi_t\Sigma^{-3}E\llbracket\GG_2/{\GG_2^1}\rrbracket)\cong 
H^s(\tmfg, \pi_{t+3}E) \llbracket \GG_2/{\GG_2^1} \rrbracket.
\]
For the $E_2$ page see Figure \ref{figure:tmf-shifted} (the $E_2$ page is 24-periodic with respect to the  $t-s$ axis), and for the $E_{\infty}$ page, see, for example, \cite[Fig. 4]{BG}. From these we deduce that $y_n$ with the property $y_n\kappabar\eta\neq 0$ (and having filtration less than 5) must be detected by an element in filtration zero, namely in \[H^0(\tmfg, E_{48+192n}) \llbracket \GG_2/{\GG_2^1} \rrbracket ,\] hence 
\[
   y_n=\Delta^{2+8n} g_n(j)
\]
for some $g_n(j).$ Then the condition $y_n\kappabar\eta\neq 0$ guarantees that $g_n(0)\neq 0$ mod $(2)$ (see Theorem 4.6 in \cite{BG}).
\end{proof}

\begin{proposition}\label{prop:main}
There is a $K(2)$-local equivalence 
       \[F(\tmf, E^{h\Gto})\simeq \Sigma^{45}\tmf\br{\G/{\Gto}}.\]
\end{proposition}

\begin{proof}
We, again, compute with the fixed point spectral sequence \eqref{eq:main-spectral-sequence}
\[
  E_2^{s,t} =H^s(\tmfg, \pi_tF(E, \eG))\Longrightarrow \pi_{t-s}F(E, \eG)^{\tmfg}\cong \pi_{t-s}F(\tmf, \eG).
\]
   By Corollary \ref{corollary:important}, $\Delta^{2+8n}g_n(j)$ is a permanent cycle in this spectral sequence for each $n$ and $g_n(0)\neq 0.$ Then we apply Proposition \ref{prop:Delta-is-perm-cycle-for-Dtmf}.
\end{proof}

Now we are ready to prove our main theorem. Let $\xi$ be the canonical topological generator of $\ZZ_2\cong \GG_2/{\GG_2^1}$ and recall that there
 exists a fiber sequence
\begin{equation}\label{eq:psi-fiber-sequence}
      \eG \xra{\xi-1} \eG \to \Sigma L_{K(2)}S^0,
\end{equation}
where the map $\xi$ is given by the residual action of $\xi \in \GG_2/{\GG_2^1}$ on $\eG$.

 \begin{thmm}
  Let $n=p=2$  and let $\dtmfg$ be the maximal finite
   subgroup of the Morava stabilizer group $\GG_2.$ Then the $K(2)$-local Spanier--Whitehead dual of $\dtmf$ is \[D\dtmf=F(\dtmf, L_{K(2)}S^0)\simeq \Sigma^{44}\dtmf.\]
 \end{thmm}
\begin{proof}
We map $\tmf$ into the fiber sequence \eqref{eq:psi-fiber-sequence} to get

\begin{equation}\label{eq:first-fiber-seq}
F(\tmf, \eG) \xra{\xi-1} F(\tmf, \eG)\to \Sigma D\tmf.
\end{equation}
  
The map $\xi$ in this fiber sequence is the action of $\xi\in \GG_2/{\GG_2^1}$ on the target in the function spectrum, given by Lemma \ref{lem:Lemma5}.

Now consider the fiber sequence (see \cite{Beh}, Lemma 2.3.8)
\begin{equation}\label{eq:second-fiber-seq}
   \Sigma^{45}\tmf \llbracket\ZZ_2\rrbracket \xra{\tau-1} \Sigma^{45}\tmf\llbracket\ZZ_2\rrbracket \ra \Sigma^{45}\tmf,
\end{equation}
where the map $\tau$ is given by the action of the canonical generator $\tau \in \ZZ_2$ on $\ZZ_2$.
By Proposition \ref{prop:main} the first two terms in the fiber sequences \eqref{eq:first-fiber-seq} and \eqref{eq:second-fiber-seq} are equivalent and by 
 Lemma \ref{lem:Lemma5} the maps $\xi-1$ and $\tau-1$ are equivalent, hence the cofibers are equivalent as well and we have 
$D\tmf \simeq \Sigma^{44}\tmf.$

Now we use Lemma 1.37 from  \cite{BG} which implies that there is a $\Gal(\FF_4/{\FF_2})$-equivariant equivalence
\[
   \Gal(\FF_4/{\FF_2})_{+}\wedge \dtmf \xra{\simeq} \tmf
\]
and get \[
           D\dtmf \simeq \Sigma^{44}\dtmf.
           \qedhere
        \]    
\end{proof}

\section{Spanier--Whitehead dual of $E^{hF}$ for $F\subseteq G_{48}$}
\begin{lemma}
Let $F$ be any finite subgroup of $\GG_2.$ Then we have an equivalence
$$
 DE^{hF} \simeq (DE)^{hF}.
$$
\end{lemma}
\begin{proof}
The proof goes exactly the same way as for Lemma \ref{lem:Tate_vanishes}. Since  $F$ is finite, the Tate spectrum vanishes, and we have $E^{hF}\simeq E_{hF}.$ Then
\[
 F(E^{hF}, E^{h\mathbb{G}_2})\simeq F(E_{hF}, E^{h\mathbb{G}_2}) \simeq F(E, E^{h\mathbb{G}_2})^{hF}.
 \qedhere
\]
 \end{proof}
The results of the previous section can be reformulated as the following statement.
\begin{cor}
   In the homotopy fixed point spectral sequence 
   \[
      E_2^{s,t}=H^s(\dtmfg, \pi_tDE) \Longrightarrow \pi_{t-s}(DE)^{h\dtmfg}\simeq D\dtmf
   \]
the class $\Delta^2\in H^0(\dtmfg, \pi_{44}DE)\cong H^0(\dtmfg, \pi_{48}E)$ is a permanent cycle.
\end{cor}
Using this we can now compute $DE^{hF}$ for various subgroups $F \subseteq \dtmfg$. 
\begin{thm}
   Let $F\subseteq G_{48} \subset \GG_2.$ There is a $K(2)$-local equivalence 
   \[
   DE^{hF} \simeq \Sigma^{44}E^{hF}.
\]
\end{thm}
\begin{proof}
For a subgroup $F\subseteq \dtmfg,$ let $\theta_F$ denote the inclusion map $F\xra{\theta_F} \dtmfg.$ 
It induces a map of spectral sequences 
$$\xymatrix{
   E_2^{*,*}\cong H^*(\dtmfg, \pi_*DE)\ar@{=>}[d]\ar[r]^{\theta_F} &H^*(F, \pi_*DE)\cong E_2^{*,*}\ar@{=>}[d]\\
   \pi_*(DE)^{h\dtmfg} \ar[r]^{\theta_F} &\pi_*(DE)^{hF}
}$$
and the inclusion of invariants 
\[
    E_2^{0,*}=H^0(\dtmfg, \pi_*\Sigma^{-4}E) \xra{\theta_F} H^0(F, \pi_*\Sigma^{-4}E)=E_2^{0,*}.
\]
Let $\Delta^2_F$ be the image of $\Delta^2\in H^0(\dtmfg, \pi_{44}DE)$ under $\theta_F$. Then $\Delta_F^2$ is also a permanent cycle in the fixed point spectral sequence on the right in the diagram above, 
\[
   E_2^{s,t}=H^s(F, \pi_tDE) \Longrightarrow \pi_{t-s}DE^{hF}.
\] 
Therefore, this spectral sequence is isomorphic to a shift of the homotopy fixed point spectral sequence $H^s(F, \pi_tE) \Rightarrow \pi_{t-s}E^{hF}$ by 44 and $\pi_*DE^{hF} \cong \pi_{*}\Sigma^{44}E^{hF}$.
Then, using the module structure of $DE^{hF}$ over the ring spectrum $E^{hF},$ we can extend the class $\Delta^2_F\in \pi_{44}DE^{hF}$ to the required equivalence.
\end{proof}

 \bibliographystyle{alpha}
\bibliography{bib}
\end{document}